\documentclass[12pt, a4]{amsart}
\usepackage{amscd, amsmath, amssymb}
\usepackage{fullpage}

\theoremstyle{plain}
\newtheorem{thm}{Theorem}[section]

\newtheorem{pro}[thm]{Proposition}
\theoremstyle{definition}
\newtheorem{defn}[thm]{Definition}
\newtheorem{ex}[thm]{Example}

\numberwithin{equation}{section}

\newcommand{\R}{\mathbb{R}}
\newcommand{\N}{\mathbb{N}}

\begin{document}

\title[Positive solutions for nonlocal elliptic systems]{Nonzero positive solutions of nonlocal elliptic systems with functional BCs}  

\date{}

\author[G. Infante]{Gennaro Infante}
\address{Gennaro Infante, Dipartimento di Matematica e Informatica, Universit\`{a} della
Calabria, 87036 Arcavacata di Rende, Cosenza, Italy}%
\email{gennaro.infante@unical.it}%

\begin{abstract} 
We discuss the existence and non-existence of non-negative weak solutions for second order nonlocal elliptic systems subject to functional boundary conditions. Our approach is based on classical fixed point index theory combined with some recent results by the author.
\end{abstract}

\subjclass[2010]{Primary 35J47, secondary 35B09, 35J57, 35J60, 47H10}

\keywords{Positive solution, nonlocal elliptic system, functional boundary condition, cone, fixed point index}

\maketitle

\section{Introduction}
There has been growing attention to the solvability of elliptic equations where nonlocal terms occur, one motivation being that these kind of equations often occur in applications. A widely studied case is the one of Kirchoff-type equations, see for example the review by Ma~\cite{ToFu}. Under Dirichlet boundary conditions (BCs) the equation
\begin{equation} \label{nonlocal}
-a\left(\int_{\Omega} |u|^p\,dx\right)\Delta u=  \lambda f(x,u),  \ x\in \Omega,
\end{equation}
has been studied, for example, by Corr\^{e}a and co-authors~\cite{Correa104, Correa204},  Jiang and Zhai~\cite{JiangZhai} and by Yan and co-authors~\cite{YanMa, YanRen, YanWang}. Here $\Omega\subset \mathbb{R}^m$, $m\geq 1$, is a domain with sufficiently smooth boundary, $a$ is a positive function, $p\in \mathbb{R}$. 

Note that in the local framework (that corresponds to $a\equiv 1$ in~\eqref{nonlocal}) one regains the classical Gelfand-type problem
$$
-\Delta u=  \lambda f(x,u),  \ x\in \Omega,\quad u |_{\partial \Omega} =0,
$$
we refer to the Introduction of~\cite{Bonanno} for a recent review on this topic.

By means of the classical Krasnosel'ski\u\i{}-Guo fixed-point theorem Sta\'nczy~\cite{Stanczy}, when $\Omega$  is a ball or an annular domain, studied the equation 
\begin{equation*}
-\Bigl(\int_{\Omega} f(u)\,dx\Bigr)^p\Delta u=\lambda  (f(u))^{q},  \ x\in \Omega,
\end{equation*}
a setting that covers, for example, the celebrated mean field equation 
\begin{equation}\label{mfq}
-\Delta u=  \lambda \dfrac{e^{u}}{\int_{\Omega}e^{u}\,dx }, \  x\in \Omega,
\end{equation}
we refer the reader to~\cite{CLMP, EGO} for further details on the equation~\eqref{mfq} in dimension~$2$.

Arcoya and co-authors~\cite{Arcoya}, by means of a Bolzano theorem, studied the equation
\begin{equation*}
-\Bigl(\int_{\Omega} g(u)\,dx\Bigr)^p\Delta u= f(u),  \ x\in \Omega,
\end{equation*}
while  Corr\^{e}a and de Morais Filho~\cite{Correa05} studied the equation
\begin{equation*} 
-\Bigl(\int_{\Omega} g(x,u)\,dx\Bigr)^p\Delta u= (f(x,u))^{q},  \ x\in \Omega,
\end{equation*}
via the Galerkin method. 

In the radial case, also by topological methods, Fija\l kowski and Przeradzki~\cite{Bogdan} and Engui\c{c}a and Sanchez~\cite{Luis} studied the equation
\begin{equation} \label{fpes}
-\Delta u=f\Bigl(u,\int_{\Omega} g(u)\,dx\Bigr);
\end{equation}
note that in \eqref{fpes} a nonlocal term occurs within the nonlinearity $f$, in what follows we shall consider a similar setting.

It is worth to mention that the fairly general equation
$$
-\mathcal{A}(x,u ) \Delta u= \lambda f(u), \  x\in \Omega,
$$
where $\mathcal{A}$ is a suitable functional defined on $\Omega \times L^{p}(\Omega )$,
 has been studied by Chipot and co-authors~\cite{Chipot09, Chipot14}, while 
Faraci and Iannizzotto~\cite{far-ian} studied the equation 
$$
-\Delta u-\mu_1 u=A[u]  f(u), \  x\in \Omega,
$$
where $\mu_1$ is the principal eigenvalue of the Laplacian with Dirichlet BCs and $A[u]$ is a nonlinear functional. The interesting case of  $p$-Laplacian equations with nonlocal terms has been recently discussed by Santos and co-authors in~\cite{Santos1,Santos2}.

The common feature of the above  mentioned problems is the requirement that the solution vanishes on the boundary of $\Omega$. Non-homogeneous BCs
in the case of nonlocal elliptic equations have been recently studied by Wang and An~\cite{WangAn} and Morbach and Corr\~{e}a~\cite{MorCor}. 

In the context of~\emph{systems} of nonlocal elliptic equations, we mention the papers by Chen and Gao~\cite{ChenGao05} and the recent paper by do \'O et al~\cite{do15}. In particular in the latter paper the authors study, in the radial case and
by topological methods, the system
\begin{equation*}
 \left\{
\begin{array}{ccc}
-a_i\left(\int_{\Omega} |u_i|^{p_{i}}\,dx\right)\Delta u_i= \lambda_i f_i(|x|, u),  & x\in \Omega, & i=1,2,\ldots,n,\\
u_i=0, & x\in \partial \Omega, & i=1,2,\ldots,n.
\end{array}
\right.
\end{equation*}

Here we adapt the arguments of~\cite{gi-tmna}, valid for~\emph{local} differential equations, in order to study the solvability of the system of second order elliptic functional differential equations subject to functional BCs
\begin{equation}
  \label{nellbvp-intro}
 \left\{
\begin{array}{ccc}
   L_i u_i=\lambda_i f_i(x,u, w_i [u]), &  x\in \Omega, & i=1,2,\ldots,n, \\
 B_i u_i=\eta_i  \zeta_i (x) h_{i} [u], & x\in \partial \Omega, & i=1,2,\ldots,n,
\end{array}
\right.
\end{equation}
where  $\Omega\subset \R^m$ ($m\geq 2$) is a bounded domain with a sufficiently smooth boundary, $L_i$ is a strongly uniformly elliptic operator, $B_i$ is a first order boundary operator, 
 $u=(u_1,\dots, u_n)$, $f_i$ are continuous functions, $\zeta_i$ are sufficiently regular functions, $w_{i}, h_{i}$ are suitable compact functionals, $\lambda_i, \eta_i$ are nonnegative parameters. The setting for the BCs covers, for example, the special cases of \emph{linear} (\emph{multi-point} or \emph{integral}) BCs of the form
\begin{equation*}
h_{i}[u]=\sum_{k=1}^{n} \sum_{j=1}^{N}\hat{\alpha}_{ijk}u_k(\omega_j),\ \text{or}\
h_{i}[u]=\sum_{k=1}^{n}\int_{\Omega}\hat{\alpha}_{ik}(x)u_k(x)\,dx,
\end{equation*}
where  $\hat{\alpha}_{ijk}$ are non-negative coefficients, $\omega_j\in \Omega$, $\hat{\alpha}_{ik}$ are non-negative continuous functions on $\overline{\Omega}$.
There exists a wide literature on these kind of BCs, we refer the reader to the reviews~\cite{Cabada1, Conti, rma, sotiris, Stik, Whyburn} and the papers~\cite{kttmna, ktejde, Pao-Wang, Picone, jw-gi-jlms}. We point that that this setting can also be applied to nonlinear, nonlocal BCs, which have seen recently attention in the framework of (local) elliptic equations, we refer the reader to the papers by Cianciaruso and co-authors~\cite{genupa}
and Goodrich~\cite{Goodrich1, Goodrich2, Goodrich3, Goodrich4}.

Here we discuss, under fairly general conditions,  the existence and non-existence of positive solutions of the system~\eqref{nellbvp-intro}. Our approach relies on classical fixed point index theory. We present some applications of the theoretical results to nonlocal elliptic systems, where we illustrate the variety of BCs that can be approached via this method. Our results are new and complement the results of~\cite{do15}, by considering non-radial cases, by allowing the presence of functional BCs and by permitting, in the non-local terms of differential equations, an interaction between all the components of the system. We also improve the results in~\cite{gi-tmna} in the case of~\emph{local} elliptic equations, by weakening the assumptions on the~BCs.

\section{Existence and non-existence results}
In what follows, for every $\hat{\mu}\in (0,1)$ we denote by $C^{\hat{\mu}}(\overline{\Omega})$ the space of all $\hat{\mu}$-H\"{o}lder continuous functions  $g:\overline{\Omega}\to \mathbb{R}$ and, for every $k\in \mathbb{N}$, we denote by $C^{k+\hat{\mu}}(\overline{\Omega})$ the space of all functions  $g\in C^{k}(\overline{\Omega})$ such that all the partial derivatives of $g$ of order $k$ are $\hat{\mu}$-H\"{o}lder continuous in $\overline{\Omega}$ (for more details see \cite[Examples~1.13 and~1.14]{Amann-rev}).

We make the following assumptions on the domain $\Omega$ and the operators $L_i$ and $B_i$ and the functions $\zeta_i$ that occur in~\eqref{nellbvp-intro}
 (see \cite[Section~4 of Chapter~1]{Amann-rev} and \cite{Lan1, Lan2})): 
\begin{enumerate}
\item $\Omega\subset \R^m$, $m\ge 2$, is a bounded domain such that its boundary $\partial \Omega$ is an $(m-1)$-dimensional $C^{2+\hat{\mu}}-$manifold for some $\hat{\mu}\in (0,1)$, such that $\Omega$ lies locally on one side of $\partial \Omega$ (see \cite[Section 6.2]{zeidler} for more details).
\item $L_i$ is a the second order elliptic operator given by
\begin{equation*}
L_i u(x)=-\sum_{j,l=1}^m a_{ijl}(x)\frac{\partial^2 u}{\partial x_j \partial x_l}(x)+\sum_{j=1}^m a_{ij}(x) \frac{\partial u}{\partial x_j} (x)+a_i (x)u(x), \quad \mbox{for $x\in \Omega$,}
\end{equation*}
where $a_{ijl},a_{ij},a_i\in C^{\hat{\mu}}(\overline{\Omega})$ for $j,l=1,2,\ldots,m$, $a_i(x)\ge 0$ on $\bar{\Omega}$, $a_{ijl}(x)=a_{ijl}(x)$ on $\bar{\Omega}$ for $j,l=1,2,\ldots,m$.  Moreover $L_i$ is strongly uniformly elliptic; that is, there exists $\bar{\mu}_{i0}>0$ such that 
$$\sum_{j,l=1}^m a_{ijl}(x)\xi_j \xi_l\ge \bar{\mu}_{i0} \|\xi\|^2 \quad \mbox{for $x\in \Omega$ and $\xi=(\xi_1,\xi_2,\ldots,\xi_m)\in\R^m$.}$$

\item $B_i$ is a boundary operator given by
$$B_i u(x)=b_i(x)u(x)+\delta_i \frac{\partial u}{\partial \nu}(x) \quad \mbox{for $x\in\partial \Omega$},$$
where $\nu$ is an outward pointing and nowhere tangent vector field on $\partial \Omega$ of class $C^{1+\hat{\mu}}$ (not necessarily a unit vector field),  $\frac{\partial u}{\partial \nu}$ is the directional derivative of $u$ with respect to $\nu$,  $b_i:\partial \Omega \to \R$ is of class $C^{1+\hat{\mu}}$ and moreover one of the following conditions holds:
\begin{enumerate}
\item $\delta_i =0$ and $b_i(x)\equiv 1$ (Dirichlet boundary operator).
\item $\delta_i =1$, $b_i(x)\equiv 0$ and $a_i(x)\not\equiv 0$ (Neumann boundary operator).
\item $\delta_i =1$, $b_i(x)\ge 0$ and $b_i(x)\not\equiv 0$ (Regular oblique derivative boundary operator).
\end{enumerate}

\item $\zeta_i \in  C^{2-\delta_i+\hat{\mu}}(\partial\Omega)$. 

\end{enumerate}

Under the previous conditions (see  \cite{Amann-rev}, Section 4 of Chapter 1) a strong maximum principle holds,
given $g\in C^{\hat{\mu}}(\bar{\Omega})$, the BVP 
 \begin{equation}
  \label{eqelliptic}
 \left\{
\begin{array}{ll}
   L_{i}u(x)=g(x), &  x\in \Omega, \\
 B_{i}u(x)=0, & x\in \partial \Omega,
\end{array}
\right.
\end{equation}
admits a unique classical solution $u\in C^{2+\hat{\mu}}(\bar{\Omega})$ and, moreover, given $\zeta_i \in  C^{2-\delta_i+\hat{\mu}}(\partial\Omega)$  the BVP 
 \begin{equation}
  \label{eqellipticbc}
 \left\{
\begin{array}{ll}
   L_{i}u(x)=0, &  x\in \Omega, \\
 B_{i}u(x)=\zeta_i(x), & x\in \partial \Omega,
\end{array}
\right.
\end{equation}
also admits a unique solution $\gamma_i \in C^{2+\hat{\mu}}(\bar{\Omega})$.

We recall that a \emph{cone} $P$ of a
real Banach space $X$ is a closed set with $P+P\subset P$,
  $\lambda P\subset P$ for all $\lambda\ge 0$ and  $P\cap(-P)=\{0\}$. A cone $P$ induces a partial ordering in
$X$ by means of the relation
$
x\le y \  \mbox{if and only if $y-x\in P$}.
$
The cone $P$ is {\it normal} if there exists $d>0$ such that for all $x, y\in X$ with $0 \le x\le y$ then
$\|x\|\le d \|y\| $.
Note that every cone $P$ has the Archimedean property; that is, $n x\le y$ for all $n\in \N$ and some $y\in X$ implies $x\le 0$. In what follows, with abuse of notation, we will use the same symbol ``$\ge$" for the different cones appearing in the paper.

In order to seek solutions of the system~\eqref{nellbvp-intro}, we make use of the cone of non-negative functions $P=C(\bar{\Omega},\R_+)$. The solution operator associated to the BVP~\eqref{eqelliptic}, $K_i:C^{\hat{\mu}}(\bar{\Omega})\to C^{2+\hat{\mu}}(\bar{\Omega})$, defined as $K_{i}g=u$ is linear and continuous. It is known (see \cite{Amann-rev}, Section 4 of Chapter 1) that $K_{i}$ can be extended uniquely to a continuous, linear and compact operator (that we denote again by the same name) $K_{i}:C(\bar{\Omega})\to C(\bar{\Omega})$ that leaves the cone $P$ invariant, that is $K_{i}(P)\subset P$.
We denote by $r(K_{i})$ the spectral radius of $K_{i}$. 
It is known (for details see Lemma~3.3 of~\cite{Lan2}) that  $r(K_{i})\in (0,+\infty)$ and there exists 
$\varphi_{i}\in P\setminus \{0\}$  such that 
\begin{equation}\label{igenfun}
\varphi_{i}=\mu_{i}K_{i}\varphi_{i},
\end{equation}{}
where $\mu_{i}=1/r(K_{i})$.

We utilize the space $C(\bar{\Omega},\R^n)$, endowed with the norm $\|u\|:=\displaystyle\max_{i=1,2,\ldots,n} \{\|u_i\|_{\infty}\}$, where $\|z\|_{\infty}=\displaystyle\max_{x\in \bar{\Omega}}|z(x)|$, and consider (with abuse of notation) the cone $P=C(\bar{\Omega},\R^n_+)$. Given $J=\prod_{i=1}^n J_i\subset \R^n_+,$ where each $J_i\subset \R$ is a closed nonempty interval, we define
$$P_J=\{u\in P: u(x)\in J \ \text{for all}\ x\in \bar{\Omega}\}.$$

We now fix $I=\prod_{i=1}^n [0,\rho_i]$, where $\rho_i\in (0,+\infty)$, and rewrite the elliptic system~\eqref{nellbvp-intro} as a fixed point problem by considering the operators $T,\Gamma:C(\bar{\Omega}, I) \to  C(\bar{\Omega},\R^n)$ given by 
\begin{align}
T(u):=(\lambda_i K_i F_i(u))_{i=1..n},\quad
\Gamma (u):=(\eta_i \gamma_i h_{i}[u] )_{i=1..n},
\end{align}
where $K_i$ is the above mentioned extension of the solution operator associated to~\eqref{eqelliptic},
 $\gamma_i \in C^{2+\hat{\mu}}(\overline{\Omega})$ is the unique solution of the BVP~\eqref{eqellipticbc} and
$$F_i(u)(x):=f_i(x,u(x), w_i[u]),\ \text{for}\ u\in C(\bar{\Omega}, I)\ \text{and}\ x\in \bar{\Omega}.$$

\begin{defn} We say that $u\in C(\bar{\Omega}, J)$ is a {\it weak solution} of the system~\eqref{nellbvp-intro} if and only if $u$ is a fixed point of the operator $T+\Gamma$, that is, 
$$
u=Tu+\Gamma u=(\lambda_i K_i F_i(u)+\eta_i \gamma_i h_{i}[u] )_{i=1..n};
$$
if, furthermore, the components of $u$ are non-negative with $u_j\not\equiv 0$ for some $j$ we say that $u$ is a \emph{nonzero positive solution} of the system~\eqref{nellbvp-intro}.
\end{defn}

For our existence result, we make use of the following Proposition that states the main properties of the classical fixed point index, for more details see~\cite{Amann-rev, guolak}. In what follows the closure and the boundary of subsets of a cone $\hat{P}$ are understood to be relative to~$\hat{P}$.

\begin{pro}\label{propindex}
Let $X$ be a real Banach space and let $\hat{P}\subset X$ be a cone. Let $D$ be an open bounded set of $X$ with $0\in D\cap \hat{P}$ and
$\overline{D\cap \hat{P}}\ne \hat{P}$. 
Assume that $T:\overline{D\cap \hat{P}}\to \hat{P}$ is a compact operator such that
$x\neq Tx$ for $x\in \partial (D\cap \hat{P})$. Then the fixed point index
 $i_{\hat{P}}(T, D\cap \hat{P})$ has the following properties:
 
\begin{itemize}

\item[$(i)$] If there exists $e\in \hat{P}\setminus \{0\}$
such that $x\neq Tx+\lambda e$ for all $x\in \partial (D\cap \hat{P})$ and all
$\lambda>0$, then $i_{\hat{P}}(T, D\cap \hat{P})=0$.

\item[$(ii)$] If $Tx \neq \lambda x$ for all $x\in
\partial  (D\cap \hat{P})$ and all $\lambda > 1$, then $i_{\hat{P}}(T, D\cap \hat{P})=1$.

\item[$(iii)$] Let $D^{1}$ be open bounded in $X$ such that
$(\overline{D^{1}\cap \hat{P}})\subset (D\cap \hat{P})$. If $i_{\hat{P}}(T, D\cap \hat{P})=1$ and $i_{\hat{P}}(T,
D^{1}\cap \hat{P})=0$, then $T$ has a fixed point in $(D\cap \hat{P})\setminus
(\overline{D^{1}\cap \hat{P}})$. The same holds if 
$i_{\hat{P}}(T, D\cap \hat{P})=0$ and $i_{\hat{P}}(T, D^{1}\cap \hat{P})=1$.
\end{itemize}
\end{pro}

With these ingredients we can now state a result regarding the existence of positive solutions for the system~\eqref{nellbvp-intro}, that extends the results of Theorem~2.4 of~\cite{gi-tmna} to this new setting. In the sequel we denote by $\hat{1}$ the constant function equal to 1 on $\bar{\Omega}$.
\begin{thm}\label{thmsol}Let  $I=\prod_{i=1}^n [0,\rho_i]$  and assume the following conditions hold. 
\begin{itemize}

\item[(a)] For every $i=1,2,\ldots,n$, $w_{i}: P_I \to \mathbb{R}$ is continuous,
$$-\infty<\underline{w}_i :=\inf_{u\in P_I }w_{i}[u]\ \text{and}\ \overline{w}_i :=\sup_{u\in P_I }w_{i}[u]<+\infty.$$
\item[(b)] 
For every $i=1,2,\ldots,n$,  $f_i\in C(\bar{\Omega}\times I\times [\underline{w}_i, \overline{w}_i ])$ and $f_i\geq 0$. Set
$$
M_i:=\max_{(x,u,w)\in\bar{\Omega}\times I \times [\underline{w}_i, \overline{w}_i ] } f_i (x,u,w).
$$
\item[(c)] There exist $\delta \in (0,+\infty)$, $i_0\in \{1,2,\ldots,n\}$ and $\rho_0 \in (0,\displaystyle\min_{i=1..n}{\rho_i})$ such that $$f_{i_0}(x,u,w)\ge \delta u_{i_0},\ \text{for every}\ (x,u,w)\in \bar{\Omega}\times I_{0}\times  [\underline{w}_{i_{0}}^{0}, \overline{w}_{i_{0}}^{0} ],$$
where $I_{0}:=\prod_{i=1}^n [0,\rho_0]$ and
$$\underline{w}_{i_{0}}^{0} :=\inf_{u\in P_{I_{0}}}w_{i_{0}}[u] \leq \overline{w}_{i_{0}}^{0} :=\sup_{u\in P_{I_{0}}}w_{i_{0}}[u].$$

\item[(d)] For every $i=1,2,\ldots,n$, $\zeta_i \in C^{2-\delta_i+\hat{\mu}}(\partial\Omega)$, $\zeta_i\geq 0$, 
$h_{i}: P_I \to [0, +\infty )$ is continuous and bounded. We set 
$$\overline{h}_i:=\sup_{u\in \partial P_I }h_{i}[u].$$

\item[(e)]  For every $i=1,2,\ldots,n$ the following two inequalities are satisfied
\begin{equation*}
\frac{\mu_{i_{0}}}{\delta}\leq \lambda_{i_0}\  \text{and}\ 
\lambda_i M_i  \| K_i(\hat{1})\|_{\infty} + \eta_i \overline{h}_i\|\gamma_i\|_{\infty} \leq \rho_i.
\end{equation*}
\end{itemize}

Then the system~\eqref{nellbvp-intro} has a nonzero positive weak solution $u$ such that $$\rho_0\leq \|u\|\ \text{and}\ \|u_i\|_{\infty}\leq \rho_i, 
\ \text{for every}\ i=1,2,\ldots,n.$$
\end{thm}

\begin{proof}
Due to the assumptions above the operator $T+\Gamma$ maps $P_I$ into $P$ and is compact (by construction, the map $F$ is continuous and bounded and $\Gamma$  is a finite rank operator). If $T+\Gamma$ has a fixed point on $\partial {P_I} \cup \partial {P_{I_{0}}}$ we are done. If $T+\Gamma$ is fixed point free on $\partial {P_I} \cup \partial {P_{I_{0}}}$, then the fixed point index is defined and we can make use of Proposition~\ref{propindex}.

We now show, by contradiction, that
\begin{equation*} 
\sigma  u\neq Tu+\Gamma u\ \text{for every}\ u\in \partial P_{I}\
\text{and every}\  \sigma >1.
\end{equation*}
If this false, then there exist $u\in \partial P_{I}$ and $\sigma >1$ such that $\sigma  u= Tu+\Gamma u$. Since $u\in \partial P_{I}$ there exists $j$ such that $\| u_j\|_{\infty} = \rho_j$ and for every 
for every $x\in \overline{\Omega}$ we have
\begin{multline}\label{idx1in}
\sigma u_{j}(x)=
\lambda_j K_jF_j(u)(x) + \eta_j h_{j} [u]\gamma_j(x)\leq \| \lambda_j K_jF_j(u) + \eta_j h_{j} [u]\gamma_{j}\|_{\infty}\\
\leq \|\lambda_j K_j(M_j\hat{1}) \|_{\infty}+ \|\eta_j \overline{h}_{j}\gamma_{j}\|_{\infty} = \lambda_j M_j  \| K_j(\hat{1})\|_{\infty} + \eta_j \overline{h}_{j}\|\gamma_j\|_{\infty} \leq \rho_j.
\end{multline}
Passing to the supremum for $x\in \overline{\Omega}$ in~\eqref{idx1in} we get $\sigma  \rho_j\leq \rho_j$, a contradiction. By $(ii)$ of Proposition~\ref{propindex} we obtain $$i_{P}(T+\Gamma, P_I\setminus \partial {P_I} )=1.$$

We now show, by contradiction, that
\begin{equation*} 
u\neq Tu+\Gamma u +\sigma \varphi \ \text{for every}\ u\in \partial P_{I_0}\ \text{and every}\ \sigma  >0,
\end{equation*}
where $\varphi=(\varphi_1,\ldots,\varphi_n)$ and $\varphi_i$ given by~\eqref{igenfun}.

If this does not hold, there exists $u\in \partial P_{\rho_0}$ and $\sigma  >0$ such that
$
u= Tu+\Gamma u+\sigma  \varphi . 
$
Then we have $u\ge \sigma  \varphi$ and, in particular, $u_{i_0}\ge \sigma  \varphi_{i_0}$. For every $x\in \bar{\Omega}$ we have
\begin{multline*}
u_{i_0}(x)=(\lambda_{i_0}K_{i_0} F_{i_0} u)(x)+\eta_{i_0}  h_{i_0}[u]\gamma_{i_0}(x)   +\sigma  \varphi_{i_0}(x) \\
\ge (\lambda_{i_0}K_{i_0} \delta u_{i_0})(x) +\sigma  \varphi_{i_0}(x)\ge 
(\lambda_{i_0} \delta  K_{i_0}(\sigma  \varphi_{i_0}))(x) +\sigma  \varphi_{i_0}(x)\\
=  \frac{ \sigma  \lambda_{i_0} \delta}{\mu_{i_{0}}} \varphi_{i_0}  (x) +\sigma  \varphi_{i_0}(x)  \geq 2\sigma  \varphi_{i_0}(x).
\end{multline*}
By iteration we obtain, for $x\in \bar{\Omega}$,
$$
u_{i_0}(x)\ge n\sigma  \varphi_{i_0}(x) \ \text{for every}\ n\in\mathbb{N},
$$
which contradicts the boundedness of $u$. This gives, by $(i)$ of Proposition~\ref{propindex}, that $$i_{P}(T+\Gamma, P_{I_0} \setminus \partial {P_{I_{0}}})=0.$$

By means of $(iii)$ of Proposition~\ref{propindex}, $T+\Gamma$ has a fixed point in $P_I\setminus (\partial {P_I}\cup {P_{I_0}})$.
\end{proof}
We now provide a non-existence result which extends Theorem~2.7 of~\cite{gi-tmna}.
\begin{thm}\label{nonexthm}
Let $I=\prod_{i=1}^n [0,\rho_i]$ and 
assume that
for every $i=1,2,\ldots,n$ we have:
\begin{itemize}
\item
$f_i\in C(\bar{\Omega}\times I\times [\underline{w}_i, \overline{w}_i ])$ and there exist $\tau_i \in (0,+\infty)$ such that
$$
0\leq f_i (x,u,w)\leq \tau_i u_i,\ \text{for every}\ (x,u,w)\in\bar{\Omega}\times I\times [\underline{w}_i, \overline{w}_i ],
$$
\item
$\zeta_i \in C^{2-\delta_i+\hat{\mu}}(\partial\Omega)$, $\zeta_i\geq 0$, $h_{i}: P_I \to [0, +\infty )$ is continuous and there exist $\theta_i \in (0,+\infty)$
and
$$h_i[u]\leq \theta_i  \|u\|, \ \text{for every}\ u \in P_I,$$
\item  the following inequality holds
\begin{equation}\label{nonexineq}
  \lambda_i \tau_i   \| K_i(\hat{1})\|_{\infty} + \eta_i \theta_i   \|\gamma_i\|_{\infty}<1.
 \end{equation}
\end{itemize}
Then the system~\eqref{nellbvp-intro} has at most the zero solution in $P_{I}$. 
\end{thm}
\begin{proof}
Suppose, for the sake of contradiction, there exists $u\in P_{I}\setminus\{0\}$ such that $u$ is a fixed point for $T+\Gamma$. 
Then there exists $j$ such that  $\|u_{j}\|_{\infty}=\sigma>0$. Then we have 
\begin{multline}\label{nonext}
u_{j}(x)=
\lambda_j K_jF_j(u)(x) + \eta_j h_{j} [u]\gamma_j(x)\leq \| \lambda_j K_jF_j(u) + \eta_j h_{j} [u]\gamma_j\|_{\infty}\\
\leq \|\lambda_j K_j(\tau_j \sigma \hat{1}) \|_{\infty}+ \|\eta_j \theta_j  \sigma\gamma_j\|_{\infty} = ( \lambda_j \tau_j   \| K_j(\hat{1})\|_{\infty} + \eta_j \theta_j  \|\gamma_j\|_{\infty}) \sigma < \sigma.
\end{multline}
By taking the supremum in~\eqref{nonext} for $x\in \bar{\Omega}$ we obtain $\sigma  < \sigma$, a contradiction.
\end{proof}

\section{Two examples}
In the this last Section we show the applicability of results above in the context of systems of nonlocal elliptic equations with functional BCs.
We begin by illustrating Theorem~\ref{thmsol}.
\begin{ex}\label{exex}
Take $\Omega=\{x\in \mathbb{R}^2 : \|x\|_2<1\}$ and consider the system
\begin{equation} \label{example}
\left\{
\begin{array}{ll}
-\Bigl(\int_{\Omega}e^{|(u_1,u_2)|}\,dx\Bigr)\Delta u=\lambda_1 e^{|(u_1,u_2)|},& \text{in }\Omega , \\
-\Bigl(\int_{\Omega}e^{u_1+u_2}\,dx\Bigr)\Delta v=\lambda_2 |(u_1,u_2)|^2, & \text{in }\Omega , \\
u_1=\eta_1h_{1}[(u_1,u_2)],\ u_2=\eta_2h_{2}[(u_1,u_2)], & \text{on }\partial \Omega ,%
\end{array}%
\right. 
\end{equation}%
where $|(u_1,u_2)|=\max \{ |u_1|,|u_2|\}$,
$$
h_{1}[(u_1,u_2)]=(u_1(0))^2+(u_2(0))^\frac{1}{2}\ \mbox{and} \ h_{2}[(u_1,u_2)]=(u_1(0))^\frac{1}{4}+\Bigl(\int_{\Omega}u_2(x)\,dx\Bigr)^2.
$$
We wish to show that, under some algebraic conditions on the parameters $\lambda_i$ and $\eta_i$,  the system~\ref{example} has a nonzero positive solution of norm less or equal to~$1$. 

First of all note that $K(\hat{1})=\frac{1}{4}(1-x_1^2-x_2^2)$, where  $x=(x_1,x_2)$, and  $\|K(\hat{1})\|_{\infty}=\frac{1}{4}$. We fix $\rho_1,\rho_2=1$ and consider 
\begin{align*}
f_1(u_1,u_2, w_1[u,u_2]):=\lambda_1 e^{|(u_1,u_2)|}w_1[u_1,u_2], \ \text{where}\ w_1[u_1,u_2]=\Bigl(\int_{\Omega}e^{|(u_1,u_2)|}\,dx\Bigr)^{-1}, \\ 
f_2(u_1,u_2, w_2[u,u_2]) :=\lambda_2 |(u_1,u_2)|^2w_2[u_1,u_2],\ \text{where}\ w_2[u_1,u_2]=\Bigl(\int_{\Omega}e^{u_1+u_2}\,dx\Bigr)^{-1}.
\end{align*}
In this case we have  $$[\underline{w}_1, \overline{w}_1 ]=[{(e\pi)}^{-1},{\pi}^{-1}], [\underline{w}_2, \overline{w}_2 ]=[{(e^2\pi)}^{-1},{\pi}^{-1}]$$
and therefore we get
\begin{align*}
M_1 =f_1(1,1, {\pi}^{-1})=e/ \pi,&\quad \bar{h}_{1}\leq 2,\\ 
M_2=f_2(1,1, {\pi}^{-1}) =1/ \pi,&\quad \bar{h}_{2}\leq 1+\pi^2.
\end{align*}
Furthermore note that  $f_1$ satisfies the condition~$(c)$ in Theorem~\ref{thmsol} for $\rho_0$ sufficiently small, due to the behaviour near the origin. 

By Theorem~\ref{thmsol} the system~\eqref{example} has a nonzero positive solution $(u_1,u_2)$
such that $0<\|(u_1,u_2)\|\le 1$ for every $\lambda_1,\lambda_2, \eta_1, \eta_2>0$ with
\begin{equation}\label{ineqex1}
\frac{ e \lambda_1}{4 \pi}+ 2 \eta_1 \leq 1\ \text{and}\ \frac{  \lambda_2}{4 \pi}+ (1+\pi^2) \eta_2 \leq 1.
\end{equation}
The inequality~\eqref{ineqex1} is satisfied, for example, when $\lambda_1=1/3, \eta_1=1/4, \lambda_2=1/5, \eta_2=1/15$.
\end{ex}
In this last example we illustrate the applicability of Theorem~\ref{nonexthm}.
\begin{ex}
Take $\Omega=\{x\in \mathbb{R}^2 : \|x\|_2<1\}$ and consider the system
\begin{equation} \label{sysnonex}
\left\{
\begin{array}{ll}
-\Bigl(\int_{\Omega}e^{|(u_1,u_2)|}\,dx\Bigr)\Delta u=\lambda_1 (u_1)^2\sin(u_2) , & \text{in }\Omega , \\
-\Bigl(\int_{\Omega}e^{u_1+u_2}\,dx\Bigr)\Delta v=\lambda_2 (u_2)^4\cos(u_1), & \text{in }\Omega , \\
u_1=\eta_1h_{1}[(u_1,u_2)],\ v=\eta_2h_{2}[(u_1,u_2)], & \text{on }\partial \Omega ,%
\end{array}%
\right. 
\end{equation}%
where $h_{1}[(u_1,u_2)]=u_1(0)+(u_2(0))^2$ and $h_{2}[(u_1,v)]=u_1(0)+(u_2(0))^3$.  As in Example~\ref{exex} we set
\begin{align*}
f_1(u_1,u_2, w_1[u_1,u_2]):=\lambda_1 (u_1)^2\sin(u_2) w_1[u_1,u_2],\ \text{where}\ w_1[u_1,u_2]=\Bigl(\int_{\Omega}e^{|(u_1,u_2)|}\,dx\Bigr)^{-1}, \\ 
f_2(u_1,u_2, w_2[u_1,u_2]) :=\lambda_2 (u_2)^4\cos(u_1) w_2[u_1,u_2],\ \text{where}\ w_2[u_1,u_2]=\Bigl(\int_{\Omega}e^{u_1+u_2}\,dx\Bigr)^{-1}.
\end{align*}
Fix $I=[0, \frac{\pi}{4}]\times [0, \frac{\pi}{2}]$ and
note that in $\bar{\Omega}\times I$ we have
$$
0\leq (u_1)^2\sin(u_2)\leq \frac{\pi}{4}u_1,\ 0\leq (u_2)^4\cos(u_1)\leq \frac{\pi^3}{8} u_2.
$$

Furthermore for $(u_1,u_2)\in P_{I}$, we have
$$
0\leq h_{1}[(u_1,u_2)]\leq \bigl( \frac{\pi}{2}+1 \bigr)\|(u_1,u_2)\|, \ 0\leq h_{2}[(u_1,u_2)]\leq \bigl( \frac{\pi^2}{4}+1 \bigr)\|(u_1,u_2)\|.
$$
Thus, in this case, the condition~\eqref{nonexineq} is satisfied if we have
\begin{equation}\label{nonexineqex}
\frac{1}{4e}  \lambda_1   + \bigl( \frac{\pi}{2}+1 \bigr) \eta_1  <1\quad \text{and} \quad  \frac{\pi^2}{8e^2}\lambda_2   +  \bigl( \frac{\pi^2}{4}+1 \bigr) \eta_2 <1.
\end{equation}
Since
the trivial solution is a solution of~\eqref{sysnonex}, as long as the inequality~\eqref{nonexineqex} is satisfied (for example when $\lambda_1=1/2, \eta_1=1/3, \lambda_2=1/2, \eta_2=1/4$),
 by Theorem~\ref{nonexthm} we obtain that the system~\eqref{sysnonex} admits only the trivial solution in~$P_{I}$.
\end{ex}

\section*{Acknowledgements}
This manuscript was presented at the International Workshop on Nonlinear Dynamical Systems and Functional Analysis held in Brasilia (Brazil) in August 2018.
G.~Infante would like to thank the Workshop Organizers for their warm hospitality and generous support.
G.~Infante was partially supported by G.N.A.M.P.A. - INdAM (Italy).

\end{document}